\documentclass{amsart}
\usepackage{amssymb}
\usepackage{amsmath,amsfonts,amsthm,amstext}
\usepackage[all]{xy}
\newtheorem{theorem}{Theorem}[section]
\newtheorem{lemma}[theorem]{Lemma}
\newtheorem{prop}[theorem]{Proposition}
\newtheorem{cor}[theorem]{Corollary}
\numberwithin{equation}{section}

\newcommand{\BF}{\mathbb{F}}

\newcommand{\BZ}{\mathbb{Z}}

\def\Vightarrow#1{\smash{\mathop{\longrightarrow}\limits^{#1}}}

\begin{document}
\title{On pro-$p$ analogues of limit groups \\
via extensions of centralizers}
\author{D. H. Kochloukova, P. A. Zalesskii}
\address{
Department of Mathematics,
University of Campinas, Cx. P. 6065, 13083-970 Campinas, SP, Brazil \\
Department of Mathematics,~University of Bras\'\i lia,\\ 70910-900
Bras\'\i lia DF,~Brazil} \email{ desi@ime.unicamp.br,
pz@mat.unb.br}
\thanks{Both  authors are partially supported by
``bolsa de produtividade em pesquisa" from CNPq, Brazil.}
\date{}
\begin{abstract} We begin a study  of a pro-$p$ analogue of limit groups via
extensions of centralizers and call $\mathcal{L}$ this new class
of pro-$p$ groups. We show that the pro-$p$ groups of
$\mathcal{L}$  have finite cohomological dimension, type
$FP_{\infty}$ and non-positive Euler characteristic. Among the
group theoretic properties it is proved that they are
free-by-(torsion-free poly -procyclic) and if non-abelian do not have a
finitely generated non-trivial normal subgroup of infinite index.
Furthermore it is shown that every 2 generated pro-$p$ group in
the class $\mathcal{L}$ is either free pro-$p$ or abelian.
\end{abstract}
\maketitle

\section{Introduction}
Abstract limit groups have recently attracted a great deal of
attention and
 played a key role in the solution of the famous Tarski
problems (\cite{K-M-06}-\cite{K-M1-05}, \cite{S1}-\cite{S6}). V.
Remeslennikov has initiated their  study in \cite{R-89} referring
to them as $\exists$-free groups, reflecting the fact that these
groups  have the same existential theory as a free group, or
$\omega$-residually free groups. O. Kharlampovich and A. Myasnikov
have studied abstract limit groups extensively under the name
fully residually free groups, i.e. groups $G$ such that for every
finite subset $X$ of $G$ there is a free group $F$ together with a
homomorphism of groups $\varphi : G \to F$ whose restriction to
$X$ is injective (see \cite{khamya} and \cite{khamya2}). In fact,
abstract limit groups $G$ are exactly  finitely generated   fully
residually free groups. The most relevant definition for this
paper is that abstract limit groups are finitely generated
subgroups of a group obtained by a finite sequence of extensions
of centralizers starting with a free group of finite rank, i.e.
every limit group can be viewed as a finitely generated subgroup
of an inductively constructed  group $B_n=B_{n-1}*_C A$ by forming
a free amalgamated product of a group $B_{n-1}$ and a free finite
rank abelian group $A$ with a cyclic amalgamation $C$ that is
self-centralized in $B_{n-1}$ and the class $B_0$ contains
precisely the finitely generated free groups.

The objective of this paper is to begin the study of a pro-$p$
analogue of limit groups. We define a class of pro-$p$ groups
$\mathcal{L}$ as the class that contains the finitely generated
pro-$p$ subgroups of a sequence of  extensions of centralizers
performed in the category of pro-$p$ groups starting with a
finitely generated free pro-$p$ group, see Section \ref{defin}.
When the term finitely generated is used for a pro-$p$ group we
mean of course topologically finitely generated. Note that the
geometric approach of Z. Sela \cite{S1}  is not available in the
pro-$p$ case and we do not know whether the class of pro-$p$
groups we define  contains only fully residually free pro-$p$
groups or contains all finitely generated fully residually free
pro-$p$ groups.

We study the group theoretic structure and the homological
properties of a pro-$p$ group in $\mathcal{L}$ and show that
various results that are known for abstract limit groups hold in
the pro-$p$ case. For example as shown in Section
\ref{sectioncentral} the pro-$p$ groups from the class
$\mathcal{L}$  are transitive commutative and the centralizer and
the normalizer of a closed procyclic subgroup always coincide and
are abelian. Furthermore every virtually soluble pro-$p$ group
from the class $\mathcal{L}$ is abelian. We list our main results
in the following

\medskip
{\bf Theorem.} {\it Let $G$ be a pro-$p$ group from the class
$\mathcal{L}$. Then
\begin{enumerate}
\item $G$ is free-by-(torsion free poly-procyclic); \item $G$ is of
finite cohomological dimension; \item $G$ is of type $FP_{\infty}$
and has non-positive Euler characteristic; \item If $G$ is
non-abelian then $G$   does not have a finitely generated
non-trivial normal subgroup of infinite index; \item If $G$ is
$2$-generated then   $G$ is either free pro-$p$ or free pro-$p$
abelian.\end{enumerate}}

\medskip
In the abstract case the same results hold. The fact that an
abstract limit group is of type $FP_{\infty}$ follows directly
from Bass-Serre theory and the abstract case of (4) is proved  in
\cite{BH}. The properties that abstract limit groups are
free-by-(torsion-free nilpotent) and of non-positive Euler
characteristic were proved in \cite{desi}.

The properties (1)-(3) stated above  are proved in different
sections of the paper: Lemma \ref{torsion}, Proposition
\ref{freenilpotent}, Corollary \ref{infinity} and Theorem
\ref{Eulerchar}. The proof of (4)  is included in Section
\ref{sectionnormal}, see Theorem \ref{normalabelian} and the proof
of (5) is the subject of Section \ref{two generated}.

We note that the methods traditionally used to prove statements
about abstract limit groups can not be used in the pro-$p$ case.
First observe that an element  of a pro-$p$ group can not be
expressed as a finite word of generators; this eliminates the
possibility to use combinatorial methods in their original sense.
The geometric version of combinatorial group theory, namely the
Bass-Serre theory of groups acting on trees, is also heavily used
in the theory of abstract limit groups. The profinite, in
particular pro-$p$ version of Bass-Serre theory, exists (see for
example \cite{horizons}) but not in full strength. This theory
developed by O. Melnikov, L. Ribes and P. Zalesskii
\cite{melnikov}, \cite{ZM}, \cite{proper}, \cite{RZ1},
\cite{Pavel1}  is one of the main tools of  our paper.

We finish the paper with a section of open questions, where we
list further possible properties of the class $\mathcal{L}$.

\section{Preliminaries on pro-$p$ groups acting on trees}
\label{preliminaries} In this section we collect properties of
pro-$p$  groups acting on pro-$p$ trees and free amalgamated
pro-$p$ products that will be used in the proofs later on. Further
information on this subject can be found in \cite{horizons}.

In accordance with \cite{horizons} the triple $(\Gamma,d_0,d_1)$
is a profinite graph, provided, $\Gamma$ is a boolean space and
$d_0,d_1:\Gamma \rightarrow \Gamma$  are continuous with
$d_i\empty d_j=d_j$ ($i,j\in\{0,1\}$). The elements in $V(\Gamma)
:=d_0 (\Gamma) \cup d_1 (\Gamma)$ and $E( \Gamma):=\Gamma
\setminus V( \Gamma) $ are called the vertices and  edges of
$\Gamma$ respectively, and, for $e\in E( \Gamma)$, $d_0(e)$ and
$d_1 (e)$ are called the initial and terminal vertices of $e$.
Note that $d_0, d_1$ are the identity when restricted to
$V(\Gamma)$. When there is no danger of confusion, we shall simply
write $\Gamma$ instead of $(\Gamma,d_0,d_1)$.

Let $(E^*,*)= (\Gamma/V(\Gamma), *)$ be a pointed profinite
quotient space with $V(\Gamma)$ as a distinguished point. Let
$\BF_p[[E^*(\Gamma),*]]$ and  $\BF_p[[V(\Gamma)]]$ be free
profinite  $\BF_p$-modules over the pointed profinite space
$(E^*(\Gamma),*)$ and over the profinite space $V(\Gamma)$ (cf.
\cite{PavelRibesbook}). Then we have the following complex of free
profinite $\BF_p$-modules
\begin{equation} \label{singularcomplex}
 0 \longrightarrow \BF_p[[E^*(\Gamma),*]] \Vightarrow{\delta}
\BF_p[[V(\Gamma)]] \Vightarrow{\epsilon} \BF_p
 \longrightarrow 0,
\end{equation} where $\delta(e)=d_1(e)-d_0(e)$ for all
$e\in E^*(\Gamma)$ and $\epsilon(v)=1$ for all $v\in V(\Gamma)$.

The profinite graph $\Gamma$ is {\it connected} if
(\ref{singularcomplex}) is exact in the middle. A {\it connected
component} of $\Gamma$ is a maximal profinite subgraph that is
connected. The graph $\Gamma$ is called a {\it pro-$p$ tree} if
(\ref{singularcomplex}) is exact. We say that  a pro-$p$ group $G$
acts on the pro-$p$ tree $\Gamma$ if it acts continuously on
$\Gamma$ and the action commutes with $d_0$ and $d_1$. We
 denote by $G_t$ the stabilizer of $t \in V(T) \cup E(T)$ in $G$.

Let $H, M$ and $S$ be pro-$p$ groups such that $S$ embeds as a
pro-$p$ subgroup of both $H$ and $M$,  and let $G = H \amalg_S M$
be the amalgamated free pro-$p$ product. We say that this
amalgamated free pro-$p$ product is proper if $H$ and $M$ embed in
$G$. If the above decomposition of $G$ as an amalgamated free
pro-$p$ product is proper there is a pro-$p$ tree $T$ on which $G$
acts. By definition $V(T) = \{ gH \}_{g \in G} \cup \{ gM \}_{g
\in G}$ and $E(G) = \{ gS \}_{g \in G}$. For $e = gS \in E(T)$ we
have $d_0(e) = gH$ and $d_1(e) = gM$. Note that $T/G$ is just an
edge with two vertices. Furthermore if the pro-$p$ group $G$ is
finitely generated then $T$ is second countable.
\begin{theorem} \label{Ribesproper}  \cite[Theorem 3.2]{proper} A
free pro-$p$ product with procyclic amalgamation is proper.
\end{theorem}
\begin{theorem} \label{freeaction} \cite[Prop.~3.5,~Cor.~3.6]{horizons} Let
$G$ be a pro-$p$ group acting on a pro-$p$ tree $T$ and $N =
\overline{\langle \{ G_v \}_{v \in V(T)} \rangle}$. Then $G/N$
acts on the pro-$p$ tree $T / N$ and $G/N$ is a free pro-$p$
group.
\end{theorem}
\begin{lemma} \label{converse} Let $G$ be a pro-$p$ group acting on a
pro-$p$ tree $T$
 such that $T /G$ is a pro-$p$ tree. Then $G = \overline{\langle \{ G_v
\}_{v \in V(T)} \rangle}$.
\end{lemma}
\begin{proof} Let $\widetilde G=\overline{\langle \{ G_v \}_{v \in V(T)}
\rangle}$. By  Theorem \ref{freeaction}
 $G/\widetilde G$ acts freely on $T/\widetilde G$. This means that we have a
free resolution of $\BF_p$ over $\BF_p[[G / \widetilde{G}]]$
$$0 \longrightarrow \BF_p[[E^*(T/\widetilde G),*]] \Vightarrow{\delta}
\BF_p[[V(T/\widetilde G)]] \Vightarrow{\epsilon} \BF_p
 \longrightarrow 0.$$
 Tensoring it with $ \widehat{\otimes}_{\BF_p[[G / \widetilde{G}]]} \BF_p$
we get the exact sequence (\ref{singularcomplex}) for $T/G$
implying that
 $H_1(G/ \widetilde G,\BF_p)=(G/\widetilde G)/[G/ \widetilde G, G/
\widetilde G] (G/\widetilde G)^p =0$. It
 follows that $G/ \widetilde G$ is trivial, as needed.\end{proof}
\begin{theorem} \label{actions} \cite[Thm.~3.18]{horizons} Let $G$ be a
pro-$p$ group acting on a pro-$p$ tree $T$. Then one of the
following holds :

a) $G$ is the stabilizer of a vertex of $T$;

b)  $G$ has a free non-abelian pro-$p$ subgroup $P$ such that $P
\cap G_v = 1$ for every vertex $v$ of $T$;

c)  There exists an edge $e$  of $T$ whose stabilizer $G_e$ is
normal in $G$ and the quotient group $G / G_e$ is solvable and
isomorphic to the one of the following groups : $\BZ_p$ or an
infinite dihedral pro-$2$ group $C_2 \amalg C_2$.
\end{theorem}
The above theorem implies immediately the following result.
\begin{cor} \label{torsionaction} Let $G$ be a torsion pro-$p$ group acting
on a pro-$p$ tree. Then $G$ fixes a vertex.
\end{cor}
\begin{theorem}  \label{Kurosh1} \cite[Thm.~5.6]{melnikov} Let $G$ be a
pro-$p$ group acting on a second countable (as topological space)
pro-$p$ tree $T$ with trivial edge stabilizers. Then $G$ is a free
pro-$p$ product of some vertex stabilizers $G_v$ and a free
pro-$p$ group $F$.
\end{theorem}
\begin{theorem} \label{strongKurosh} \cite[Thm.~B]{Pavel1} Let $H = Q
\amalg_C A$ be a proper amalgamated free pro-$p$ product and $B$
be a normal subgroup of $A$ such that $B \cap C = 1$. Let $G =
\overline{\langle B^g \rangle}_{g \in H}$. Then $G$ is the free
pro-$p$ product  of groups $B^q$
where $q$ runs over a closed set of coset representatives for $C$ in $Q$.
\end{theorem}

The above theorem was stated in \cite[Thm.~B]{Pavel1} using the
language of fundamental groups of  profinite graphs of profinite
groups, and it was observed that the statement  holds for
pro-${\mathcal C}$ groups, where ${\mathcal C}$ is a class of
finite groups  closed under subgroups, quotients and extensions.
Note that $H$ acts on the standard pro-$p$ tree $T$ associated
with the free amalgamated pro-$p$ product.
\begin{lemma}\label{component}
  Let $G$ be a profinite group acting on a profinite graph $S$ and let
$m_1, m_2$ be elements of a connected component
 $C$ of $S$. If $g\in G$ with $gm_1 =m_2$, then $g$ leaves $C$ invariant,
i.e. $g\in
 Stab_{G}(C).$ In other words the restriction of the
 factorization $S\longrightarrow S/G$ to $C$ coincides with the
 factorization $C\longrightarrow C/Stab_G(C)$.\end{lemma}
 \begin{proof} Since $m_2\in C\cap gC$, $C\cup gC$ is connected (cf.
\cite[Exercise~1.9~(i)]{horizons}),
  so $C=gC$.\end{proof}

\section{Definition of pro-$p$ analogues of limit groups \\ via extensions
of centralizers}

\label{defin} First we define recursively a class of pro-$p$
groups $\mathcal{G}_n$.  Denote by $\mathcal{G}_0$ the class of
all free pro-$p$ groups of finite rank. Define inductively the
class $\mathcal{G}_n$ of pro-$p$ groups $G_n$, where $G_n$ is a
free pro-$p$ amalgamated product $G_{n-1} \amalg_{C} A$, where
$G_{n-1}$ is any group from  the class $\mathcal{G}_{n-1}$, $C$ is
any self-centralized procyclic pro-$p$ subgroup of $G_{n-1}$, $A$
is any finite rank free abelian pro-$p$ group such that $C$ is a
direct summand of $A$.  In Lemma \ref{torsion} we will show that
every $G_n$ is torsion-free.

\medskip
{\bf Definition.} {\it The class of pro-$p$ groups $\mathcal{L}$
consists of all
 finitely generated pro-$p$ subgroups $H$ of some $G_n \in \mathcal{G}_n$
where $n \geq 0$. If $n$ is minimal one with the above properties
we say that $H$ has weight $n$.}

\medskip
{\bf Examples.} Note that all Demushkin pro-$p$ groups $H$ with
the invariant $q = \infty$ have pro-$p$ presentation $\langle x_1,
\ldots , x_d \mid [x_1,x_2] \ldots [x_{d-1}, x_d] \rangle$ where
$d$ is even. Write $H_d$ for the above group. If $d = 2$ then
$H_2$ is a free abelian pro-$p$ group of finite rank, so a pro-$p$
group of the class $\mathcal{L}$.

Suppose that $d$ is divisable by 4 and define $F_1$ as the free
pro-$p$ group with basis $x_1, \ldots , x_{d/2}$ and $F_2$ the
free pro-$p$ group with basis $x_{d/2 +1}, \ldots , x_d$. Let
$C_1$ be the procyclic subgroup of $F_1$ generated by $z_1 = [x_1,
x_2] \ldots [x_{d/2 -1}, x_{d/2}]$ and $C_2$ be the procyclic
subgroup of $F_2$ generated by $z_2 = [x_{d/2 +1}, x_{d/2 + 2}]
\ldots [x_{d -1}, x_{d}]$. Then $H \simeq F_1 \amalg_{C_1 \simeq
C_2} F_2$ where the isomorphism between $C_1$ and $C_2$ sends
$z_1$ to $z_2^{-1}$. Note that there is an isomorphism between
$F_1$ and $F_2$ that identifies $C_1$ and $C_2$ i.e. $x_1, x_2,
\ldots , x_{d/2}$ go to $x_d, x_{d-1}, \ldots , x_{d/2 + 1}$. Thus
$H \simeq F \amalg_C F$ where $F$ is a free pro-$p$ group of rank
$d/2$ and $C$ is a selfcentralized procyclic subgroup of $F$. Note
that $H \simeq F \amalg_C F$ embeds in $T = F \amalg_C A$, where
$A \simeq \BZ_p^2$ with $A / C \simeq \langle a \rangle$, since $F
\amalg_C F \simeq F \amalg_C F^a$ is the pro-$p$ subgroup of $T$
generated by $F$ and $F^a$. Thus $H$ is a pro-$p$ group of the
class $\mathcal{L}$.
\begin{lemma} \label{torsion} The groups $G_n \in \mathcal{G}_n$ are always
of finite cohomological dimension. In particular every pro-$p$
group from the class $\mathcal{L}$  has finite cohomological
dimension and so is torsion free.
\end{lemma}
\begin{proof} We induct on $n$, the case $n = 0$ is obvious. Suppose $n \geq
1$. Since $G_{n-1}$, $A$ and $C$ all have finite cohomological
dimensions, we have that for sufficiently large $i$, $i \geq i_0$
say, all cohomology groups $H^i(G_{n-1}, \BF_p)$, $H^i(A, \BF_p)$
and $H^i(C, \BF_p)$ are zero. By Theorem \ref{Ribesproper} the
free pro-$p$ product $G_n = G_{n-1} \amalg_C A$ is proper and
hence there is corresponding long exact sequence in homology and
cohomology.  Then $H^{i+1}(G_n, \BF_p) = 0$ for $i \geq i_0$ and
so by \cite[Cor.~7.1.6]{PavelRibesbook} $G_n$ has finite
cohomological dimension.
\end{proof}
\section{Free-by-(torsion free poly-procyclic) pro-$p$ groups}

In this section we prove some properties of the pro-$p$ groups of
the class $\mathcal{L}$ that are known to hold for abstract limit
groups. The facts that abstract limit groups are
free-by-(torsion-free nilpotent) and the Euler characteristic is
non-positive were proved in \cite{desi}. The fact that abstract
limit groups are of type $FP_{\infty}$ follows directly from
Bass-Serre theory but this theory does not hold for pro-$p$ groups
in general. So we find an alternative way of proving that the
pro-$p$ groups of the class $\mathcal{L}$ are of type
$FP_{\infty}$.

The next lemma is valid for free pro-$p$ products of free abelian
pro-$p$ groups that can be expressed as an inverse limit of free
pro-$p$ products of finitely many free abelian pro-$p$ groups.
This is true for example for free pro-$p$ products $ M =
\coprod_{w \in W} M_w $ of pro-$p$ groups $M_w$ indexed locally
constantly by a profinite space $W$ introduced in \cite{GR} (see
Proposition 2.1 there). This means that $W$ is a finite union of
clopen subsets $W_\alpha$ such that $M_v=M_w$ for $v,w\in
W_\alpha$.

\begin{lemma} \label{centralseries} Let  $M$ be a locally constant free
pro-$p$ product
$$
M = \coprod_{w \in W} M_w
$$
where $W$ is a profinite space and  every $M_w$ is a finite rank
free abelian pro-$p$ group. Then the quotients of the lower
central series of $M$ are torsion-free.
\end{lemma}
\begin{proof}
It suffices to prove the lemma for $W$ finite and then take
inverse limit. For $W$ finite we use the pro-p version of Magnus
embedding proved by Lazard \cite{lazard}. Let $Y_w$ be a basis of
$M_w$ as a free abelian pro-$p$ group and $X_w$ be a set of the
same cardinality as $Y_w$. Set $X$ as the disjoint union $\cup_{w
\in W} X_w$ and consider the ring of non-commutative formal power
series $R =\BZ_p[[X]]$. The closed group generated by $Y = \{ 1 +
x \}_{x \in X}$ is a free pro-$p$ group with a basis $Y$.

Let $I$ be the closed two sided ideal of $R$  generated by $\{ x_i
x_j - x_j x_i \mid x_i, x_j \in X_w, w \in W \}$ and set $S = R /
I$. Then the closed group $T$ generated by the image of $Y$ in $S$
is isomorphic to $M$ (see Theorem 5.9a in \cite{RZ1}). Consider
the closed two-sided ideal $J$ of $S$ generated by the image of the
elements of $X$ in $S$. Then the $(i-1)$-th quotient  of the lower
central series of $M$
 is isomorphic to $((T-1) \cap J^{i-1}) / ((T-1) \cap J^i)$. Note that  $((T-1) \cap
J^{i-1}) / ((T-1) \cap J^i)$ embeds in the free $\BZ_p$-module
$J^{i-1} / J^{i}$, so $((T-1) \cap J^{i-1}) / ((T-1) \cap J^i)$ is
torsion-free.
\end{proof}
We do not know  whether every $G_n \in \mathcal{G}_n$ is
residually torsion-free nilpotent pro-$p$ though abstract limit
group are residually torsion-free nilpotent \cite{desi}. Still we
can prove that $G_n$ is residually torsion-free poly-procyclic.
\begin{theorem} \label{residually poly} Every  $G_n \in \mathcal{G}_n$ is an inverse limit of torsion-free poly-procyclic groups and so is residually torsion-free
poly-procyclic.
\end{theorem}
\begin{proof} We use induction on $n$. Note that $G_0$ is a free pro-$p$
group of finite rank, hence residually torsion-free nilpotent
pro-$p$.

Suppose now that $G_n = G_{n-1} \amalg_C A$, $C \simeq \BZ_p, A =
C \times B \simeq \BZ_p^m$ and $G_{n-1}$ is residually
torsion-free poly-procyclic.  Then the intersection of the kernels
of all maps $G_{n-1} \amalg_C A \to Q \amalg_C A$,whose
restriction to $G_{n-1}$ is a projection to a poly-procyclic
torsion-free group $Q$ and whose restriction to $B$ is the
identity map, is trivial. To see this let
$\psi:G_{n-1}\longrightarrow \prod_{i=1}^{\infty}P_i$ be an
embedding of $G_{n-1}$ in a direct product of poly-procyclic pro-$p$
groups. Denote by $\psi_k:G_{n-1}\longrightarrow
\prod_{i=1}^{k}P_i$ the composite of $\psi$ with the projection
$\prod_{i=1}^{\infty}P_i\longrightarrow \prod_{i=1}^{k}P_i$ and
put $\psi_k(G_{n-1})=Q_k$. Then $G_{n-1}$ is the inverse limit of
the groups $Q_k$ and  hence the  (descending) intersection
$\bigcap B_k$ (that can be viewed as the inverse limit) of kernels
$B_k$ of $\psi_k$ is trivial. The kernel $N_k$ of the map $G_{n-1}
\amalg_C A \to Q_k \amalg_C A$ is exactly $\overline{\langle B_k^g
  \rangle}_{g \in G_n}$ and so by Theorem
\ref{strongKurosh} it is the free pro-$p$ product of groups
$B_k^q$
 where $q$ runs over a closed set of coset representatives for $
 C$ in $A$.  But the
 inverse limit of the groups $B_k$ is trivial, hence so is the inverse limit
 of the groups $N_k$.

Thus it remains to show that a proper amalgamated product $H = Q
\amalg_{{C}} A$, where $Q$ is torsion-free poly-procyclic, the
image of $C$ in $Q$ is infinite (and so identified with $C \simeq
\BZ_p$) and $C$ is a  direct summand of $A$, is residually
torsion-free poly-procyclic.

Let $\theta : H = Q \amalg_C A \to Q$ be the homomorphism that is
the identity map on $Q$ and the projection of $A$ to $C$ with
kernel $B$. Then the kernel of $\theta$ is the normal closure of
$B = ker(\theta) \cap A$ in $H$. Since $ker(\theta) \cap Q^g =
(ker(\theta) \cap Q)^g = 1$ by Theorem \ref{strongKurosh} $ ker
(\theta)$ is a free pro-$p$ product  of $B^g$ over some closed
subset $S$ of representatives of the coset classes of $Q / C$. Let
$U$ be a normal open subgroup of $Q$ and $\mu_U$ be the
homomorphism  from $ker(\theta)$ to the free pro-$p$ product
$V_U=\coprod_{g \in Q / UC} B^g$  that identifies the copies of
$B^g $ when $g$ represent the same element in the image of $S$ in
$Q / U$. Then $\cap ker(\mu_U) = 1$ and it suffices to produce a
filtration of characteristic subgroups of $V_U$ with torsion-free
finitely generated abelian quotients. By Lemma \ref{centralseries}
the lower central series of $V_U$ has this property.
\end{proof}

\begin{prop}  \label{freenilpotent} Every pro-$p$ group from the class
$\mathcal{L}$ is free-by-(torsion free poly-procyclic).

\end{prop}
\begin{proof}
It suffices to prove the proposition for $G_n \in \mathcal{G}_n$,
 by induction on $n$. The case when $n = 0$ is obvious, so
we assume that $ n\neq 0$. Suppose the statement is true for any group $\in \mathcal{G}_{n-1}$. Suppose there is a counterexample $\in \mathcal{G}_n$  and let
$G_n= G_{n-1} \amalg _C A$ be one with the minimal number of generators $d(G_n)$. Let $T$ be a complement of $C \simeq \BZ_p$ in $A \simeq \BZ_p^m$, so
 $T \simeq \BZ_p^{m-1}$. Let $D \simeq \BZ_p$ be a direct summand of $T$, i.e. $T=B \times D$ (with $B$ possibly the trivial group) and $\bar{G}_n$ be the quotient of $G_n$ by the normal closure of
 $D$. Then $\bar{G}_n \simeq G_{n-1} \amalg_C \BZ_p^{m-1}$. By
Theorem \ref{residually poly}
$\bar{G}_n=\lim\limits_{\displaystyle\longleftarrow} Q_i$ is an
inverse limit of torsion-free poly-procyclic pro-$p$ groups. Thus
for certain $i$ the restriction to $C$ of the projection
$G_n\longrightarrow Q_i$ is injective. Denote by $\varphi $ the
canonical map $\bar{G}_n \to Q_i$. By the induction hypothesis and
the minimality assumption on $d(G_n)$ there is an epimorphism $\mu
: \bar{G}_n\longrightarrow P$ to a poly-procyclic group $P$ with
the free kernel. Thus putting $Q$ to be the quotient of
$\bar{G}_{n}$ modulo $Ker(\mu) \cap Ker (\varphi)$ we have an
epimorphism $\eta:\bar{G}_n \to Q$ with free kernel whose
restriction to $C$ is injective.

 Let $\theta : G_n \to \bar G_{n}\to Q$ be the composition  of $\eta$ and the natural epimorphism $\nu:G_n \to \bar G_n$ and let $K$ be its kernel. Since $K \cap C = 1$, by
Theorem \ref{Kurosh1}  $K$ is the free pro-$p$
product of some conjugates of $K \cap G_{n-1}$, $K\cap A$ and a free pro-$p$ group $F$. But the restriction of $\nu$ to $G_{n-1}\amalg_C B$ is injective and so  $K\cap (G_{n-1}\amalg_C B)$ is naturally isomorphic to the kernel of $\eta$, which is free pro-$p$. So the kernel of $\theta_{|G_{n-1}}$ is free since it is a subgroup of $K\cap (G_{n-1}\amalg_C B)$. Furthermore, the intersection $K\cap A$ is at most procyclic. Thus $K$ is free pro-$p$.
\end{proof}
\begin{cor} \label{infinity} Every pro-$p$ group $G$ from the class
$\mathcal{L}$ is of type $FP_{\infty}$. In particular, $G$ is finitely presented.
\end{cor}
\begin{proof} Let $N$ be a closed normal subgroup of $G$ such that $Q = G /
N$ is a torsion-free poly-procyclic pro-$p$ group, in  particular is of
finite rank, and $N$ is a free pro-$p$ group. Then by
\cite[Thm. A]{King} $G$ is of type $FP_m$ if and only if
$H_i(N, \BF_p)$ is finitely generated as a $\BF_p[[Q]]$-module for
every $i \leq m$. Since $G$ is of type $FP_1$ (i.e. finitely
generated as a pro-$p$ group) $H_1(N, \BF_p)$ is finitely
generated as a $\BF_p[[Q]]$-module. Since $N$ is free pro-$p$ we
have that $H_i(N, \BF_p) = 0$ for $i \geq 2$. Thus $G$ is of type
$FP_m$ for every $m$.
\end{proof}
\begin{cor} Every non-trivial pro-$p$ group $G$ from the class $\mathcal{L}$
has infinite abelianization.
\end{cor}
\begin{proof} By Proposition \ref{freenilpotent} $G$ is residually torsion-free poly-procyclic. Any torsion-free poly-procyclic group has infinite abelianization. The result follows.
\end{proof}

\section{Normalizers and centralizers} \label{sectioncentral}
\begin{theorem} \label{centralizers}  Let  $G$ be a pro-$p$ group from the
class $\mathcal{L}$. Then for any $g \in G \setminus \{ 1 \}$
$$N_{G}(\overline{ \langle g \rangle} ) = C_{G}(\overline{\langle g
\rangle}) \hbox{ is a free abelian pro-p group of
 finite rank.}$$
\end{theorem}
\begin{proof} Note that it is sufficient to consider the case
$G = G_n \in \mathcal{G}$.
 We  argue by
induction on $n$.

\medskip
I. Suppose first that $n = 0$. Note that every closed subgroup of
a free pro-$p$ group is free pro-$p$
\cite[Cor.~7.7.5]{PavelRibesbook}. Therefore $N =
N_{G_0}(\overline{\langle g \rangle})$ is free pro-$p$ and  by
\cite[Proposition~8.6.3]{PavelRibesbook} the rank of $N$ is not
bigger than the rank of $\overline{\langle g \rangle}$ so that $N
\simeq \BZ_p$, in particular $N$ is abelian.
 Thus $N \subseteq C_{G_0} (\overline{\langle g \rangle})$, hence $C_{G_0}
(\overline{\langle g \rangle}) = N_{G_0} (\overline{\langle g
\rangle}) \simeq \BZ_p$ as required.

\medskip
II. Suppose now that $n > 0$ and that the theorem holds for
$G_{n-1}$. Thus $G = G_{n-1} \amalg_{C} A$, where $C\cong \BZ_p$,
$A\cong \BZ_p^m$ and $C$ is a direct factor of $A$. Suppose first
that $N_{G_n}(\overline{\langle g \rangle}) \not=
C_{G_n}(\overline{\langle g \rangle}) $ and choose $t \in
N_{G_n}(\overline{\langle g \rangle})  \setminus
C_{G_n}(\overline{\langle g \rangle})$. Thus $g^t = g^{\lambda}$
for some $\lambda \in \BZ_p \setminus (1 \cup p \BZ_p)$. Put
$M=\langle t,g\rangle$ and observe that it is solvable but not
abelian.

Since $G_n$ splits as a proper free pro-$p$  product with
amalgamation, $G_n$ acts on the canonical pro-$p$ tree from
Section \ref{preliminaries} and $M$ acts on the same pro-$p$ tree
by restriction. By Theorem \ref{actions} either $M = M_v$ i.e. $M$
is contained in the stabilizer $G_v$ of a vertex or there is a
stabilizer of an edge $M_e$ such that $M_e$ is a normal subgroup
of $M$ and either $M / M_e \simeq \BZ_p$ or $M / M_e \simeq C_2
\amalg C_2$ (in this case $p = 2$). If $M=M_v$ we may assume that
$M$ is a subgroup of $G_{n-1}$ or $\BZ_p^m$ and so by induction
hypothesis $M$  is abelian, a contradiction.

Let us consider the latter case. By conjugating $M_e$ if necessary
we may assume that $M_e$ is in $C$. Since $M_e$ is cyclic, by
\cite[Corollary~2.7]{RZ1}
\begin{equation} \label{edge1}
N_G(M_e)=N_{G_{n-1}}(M_e)\amalg_{C} N_A(M_e)=C\amalg_{C} A=A
\end{equation}
 and so  $M\leq N_G(M_e)=A$ is abelian, a contradiction with $M$ solvable
but not abelian. Thus we have proved that
$$N_{G_n}(\overline{\langle g \rangle}) =
C_{G_n}(\overline{\langle g \rangle}). $$

\medskip
 III. Finally it remains to show that $C_{G_n}(g)$
is a finitely generated abelian pro-$p$ group. By
\cite[Lemma~3.11]{horizons} there is a minimal pro-$p$ subtree
$T_1$ on which $C_{G_n}(g)$ acts. If $T_1$ is just one vertex,
$C_{G_n}(g)$ is conjugate to a subgroup of either $G_{n-1}$ or $A$
and so we deduce the result from the induction hypothesis.

Suppose $|T_1|>1$. Denote by $K$ the kernel of the action of
$C_{G_n}(g)$ on $T_1$. Then $K$ is a subgroup of every edge
stabilizer of the action of $C_{G_n}(g)$ on $T_1$, but any edge
stabilizer is procyclic. In particular $K$ is procyclic.

 There are two cases. First if $g \in K$ then
$g$ is in the stabilizer  $G_e$ of an edge and this case can  be
resolved  by applying (\ref{edge1}). If $g \not\in K$ note that
$C_{G_n}(g) /K$ acts on $T_1/K$ faithfully irreducibly and so by
\cite[Lemma~3.16~(b)]{horizons} $C_{G_n}(g) /K$ is free pro-$p$
containing a procyclic non-trivial  normal subgroup $\langle
g\rangle K/K$, hence procyclic (cf.
\cite[Proposition~8.6.3]{PavelRibesbook}). Then $C_{G_n}(g)$ is
procyclic-by-procyclic, which corresponds exactly to the case
treated in II; indeed, in II it was shown that the
procyclic-by-procyclic group $M$ is always abelian.
\end{proof}
\begin{cor} \label{transitivecom} The group $G_n$ is transitive commutative
i.e. if $[g,h] = 1 = [h,t]$ for some non-trivial elements $g,h,t$
of $G_n$ then $[g,t]=1$.
\end{cor}
\begin{proof} Note that $g, t \in C_{G_n}(h)$ and $C_{G_n}(h)$ is abelian by
Theorem \ref{centralizers}.
\end{proof}
\begin{cor} \label{centerless} If $G$ is a non-abelian pro-$p$ group from
the class $\mathcal{L}$ then its center is trivial.
\end{cor}
\begin{proof} If $g_1, g_2 \in G$ and $h \in Z(G) \setminus \{ 1 \}$ then
$[h, g_1] = [h, g_2] = 1$ and by transitive commutativity $[g_1,
g_2] = 1$ i.e. $G$ is abelian, a contradiction.
\end{proof}
 \begin{cor} \label{virtualabelian} Every virtually abelian pro-$p$ group
$G$ from the class $\mathcal{L}$ is abelian.
 \end{cor}
 \begin{proof} Let $H$ be a maximal abelian subgroup of $G$. If $G \not= H$
then take $g \in G \setminus H$ and note that for some $m > 1$ we
have $g^m \in H \setminus \{ 1 \}$. Then $[g, g^m] = 1 = [g^m ,
H]$ and by transitive commutativity $[g, H] = 1$ i.e. $\langle H,
g \rangle$ is an abelian subgroup of $G$, a contradiction with the
maximality of $H$.
\end{proof}
\begin{cor} \label{soluble} Every soluble pro-$p$ group $H$ from the class
$\mathcal{L}$  is abelian. If $H$  is an abelian non-procyclic
subgroup of $G_n = G_{n-1} \amalg_C A$ then $H$ is conjugate in
$G_n$ to a subgroup of $G_{n-1}$ or to a subgroup of  $A$.
\end{cor}
\begin{proof}  Let $T$ be the canonical pro-$p$ tree on which $G_n$ acts
(see Section \ref{preliminaries}). Then by Theorem \ref{actions}
from the preliminaries either $H$ stabilizes a vertex or there is
an edge $e$ such that $H / H_e$ is either $\BZ_p$ or $C_2 \amalg
C_2$.

In the first case $H$ is either conjugate to a subgroup of
$G_{n-1}$ or $H$ is a subgroup of an abelian vertex stabilizer.
Using induction on $n$ in both cases $H$ is abelian.

Suppose that $H / H_e \simeq \BZ_p$. If $H_e = 1$ there is nothing
to prove, so we can assume that $H_e \not= 1$. Since $H_e$ is
procyclic we conclude that $H_e \simeq \BZ_p$. By  Theorem
\ref{centralizers} $N_{G_n}(H_e)$ is  abelian. But  $H \subseteq
N_{G_n}(H_e)$ so $H$ is abelian. Moreover,  substituting in
(\ref{edge1}) $M$ with $H$ and conjugating $H_e$ if necessary to
have $H_e \subseteq C$ we get that  $H$ is  conjugate to a
subgroup of an abelian vertex group.

Suppose that $H / H_e \simeq C_2 \amalg C_2$. If $H_e = 1$ then
$C_2 \amalg C_2$ embeds in $H$ but by Lemma \ref{torsion} $H$ is
torsion free, a contradiction. If $H_e \not= 1$ then $H_e \simeq
\BZ_2$ and $H = N_H(H_e)$. By Theorem \ref{centralizers}
$N_H(H_e)$ is abelian and hence $H / H_e \simeq C_2 \amalg C_2$ is
abelian, a contradiction.
 \end{proof}

\section{Finitely generated  normal subgroups} \label{sectionnormal}

\begin{lemma} \label{finiteorder} Let $G$ be a pro-$p$ group with a lower
central series with torsion-free quotients and $\varphi \in
Aut(G)$ be an automorphism of finite order that acts trivially on
the  abelianization $G / [G,G]$. Then $\varphi$ is trivial.
\end{lemma}
\begin{proof} The proof is similar to the proofs of
\cite[Thm.~4.5.6]{PavelRibesbook}, \cite[Lemma,~p.323]{Lubotzky}
where the case of a finite rank free pro-$p$ group $G$ is
considered.

Let $\gamma_i(G)$ be the $i$th term of the lower central series of
$G$. We prove by induction on $i$ that $\varphi$ acts trivially on
$G / \gamma_i(G)$. Suppose we know that $\varphi(g) = g m$ for
some $g \in G, m \in \gamma_i(G)$. We aim to show that $m \in
\gamma_{i+1}(G)$. Note that since $\varphi$ acts trivially on $G /
\gamma_2(G)$ we have that $\varphi$ acts trivially on $\gamma_i(G)
/ \gamma_{i+1}(G)$. Let $k$ be the order of $\varphi$. Then $g =
\varphi^k(g) = g m \varphi(m) \varphi^2(m) \ldots \varphi^{k-1}(m)
\in g m^k \gamma_{i+1}(G)$. Then the image of $m$ in $\gamma_i(G)
/ \gamma_{i+1}(G)$ has finite order. Finally since $\gamma_i(G) /
\gamma_{i+1}(G)$ is torsion-free we get that $m \in
\gamma_{i+1}(G)$, as required. \end{proof}
\begin{theorem} \label{directproduct}  Let $H$ be a pro-$p$ group acting on
a pro-$p$ tree $T$ with a finitely generated normal pro-$p$
subgroup $L$ such that $H / L \simeq \BZ_p$. Then there is an open
subgroup $H_1$ of $H$ containing $L$ such that $H_1 /
\widetilde{L} = (L / \widetilde{L}) \times \BZ_p$, where
$\widetilde{L} = \overline{\langle L_v \rangle}_{v \in V}$ and $V$
is the vertex set of $T$.
\end{theorem}
\begin{proof} Without loss of generality we can assume that $L \not=
\widetilde{L}$ otherwise there is nothing to prove. Let $U$ be an
open subgroup of $H$ such that $L \subseteq U$. Set $\widetilde{U}
= \overline{\langle U_v \rangle}_{v \in V}$.
 By Theorem \ref{freeaction}  the group $H / \widetilde{U}$ acts on the
pro-$p$ tree $T / \widetilde{U}$ and since $U / \widetilde{U}$
acts freely, $U / \widetilde{U}$ is a free pro-$p$ group of finite
rank. Let $V_U$ be the maximal subgroup of $H$ such that $U
\subseteq V_U$ and $V_U / \widetilde{U}$ is a free pro-$p$ group.
Then $H / \widetilde{U} = (V_U / \widetilde{U}) \rtimes C_d$,
where $[H : V_U ] = d$ and $C_d$ denotes the cyclic group of order
$d$. Define $V = \cap_U V_U$.

\bigskip
{\bf Claim.} $V = L.$

\medskip
{\it Proof of the Claim. } Assume that $V \not= L$.  Then $V$ is
an open subgroup of $H$ and $V_U = V$ for infinitely many $U$ such
that $\cap U = L$. Consider the normal subgroup $L \widetilde{U} /
\widetilde{U}$ of the free pro-$p$ group  $V / \widetilde{U}$,
where $V / \widetilde{U}$ has rank at most $d(L) + 1$. Note that
$L \widetilde{U} / \widetilde{U}$ is finitely generated  and
normal in a finitely generated free pro-$p$ group. By
\cite[Proposition 8.6.13]{PavelRibesbook} either $L \widetilde{U}
/ \widetilde{U}$ is trivial or of finite index in $V /
\widetilde{U}$. In the first case  $L \subseteq \widetilde{U}$, so
$L \subseteq \cap \widetilde{U} = \widetilde{L}$, a contradiction.
Then we can assume that $L \widetilde{U} / \widetilde{U}$ has
finite index in $V / \widetilde{U}$, say $m_U$ and by Schreier's
formula $d(L \widetilde{U} /\widetilde{U}) - 1 = m_U (d(V /
\widetilde{U}) - 1)$.

Suppose that $d(V / \widetilde{U}) \not= 1$. Then $d(L) - 1 \geq
d(L \widetilde{U} / \widetilde{U}) - 1 = m_U (d(V / \widetilde{U})
- 1) \geq m_U$. Thus $m_U$ is bounded from above and so $[H : L
\widetilde{U}] = [H : V ] [V : L \widetilde{U}] = [H : V] m_U \leq
[H : V] (d(L) - 1)  < \infty$. The intersection of $L
\widetilde{U}$ over $U$ is $L \widetilde{L} = L$ and on the other
hand $L \widetilde{U}$ has bounded index in $H$ and $H / L \simeq
\BZ_p$, a contradiction.

Finally it remains to consider the case when $d(V / \widetilde{U})
= 1$.  Consider a sequence of open subgroups $U_1 \supseteq U_2
\supseteq  \ldots \supseteq \cap_i U_i = L$ such that $V_{U_i} =
V, d(V / \widetilde{U}_i) = 1$. Since $V / \widetilde{U}_{i}
\simeq \BZ_p$ is a quotient of $V / \widetilde{U}_{i+1} \simeq
\BZ_p$ we deduce that $\widetilde{U}_{i+1} = \widetilde{U}_i$.
Note that $\widetilde{L}$ is the intersection of
$\widetilde{U}_i$'s i.e. $\widetilde{L} = \cap_i \widetilde{U}_i =
\widetilde{U}_i$ . Thus $\widetilde{L} = \widetilde{U}_1 \subseteq
L \subseteq U_1$. Note that $U_1 / \widetilde{U}_1$ is a subgroup
of $V / \widetilde{U}_1 \simeq \BZ_p$, so either $U_1 /
\widetilde{U}_1$ is trivial or is $\BZ_p$. In the first case $L =
\widetilde{L}$, a contradiction and in the second case  $U_1 / L$
is a proper quotient of $U_1 / \widetilde{U}_1 \simeq \BZ_p$
(remember $L \not= \widetilde{L}$), so $U_1 / L$ is finite and a
subgroup of $H / L \simeq \BZ_p$, so $U_1 = L$, a contradiction
with $U_1$ an open subgroup of $H$. This completes the proof of
the claim.

\medskip
Let us fix one $U$ as above i.e. $U$ is an open subgroup  of $H$
that contains $L$ and write $F_U = V_U / \widetilde{U}$.
Furthermore we can choose $U$ such that  $s = [H : V_U] > 1$ and
write  $M_U = H / \widetilde{U} = F_U \rtimes C_s$.  Note that $H$
acts by conjugation on the abelianization $F_U / [F_U,F_U] \simeq
\BZ_p^{d(F_U)}$, where $d(F_U) \leq d(V_U) \leq d(L) + 1$. This
gives a homomorphism
$$\varphi_U :H \to GL_{d(F_U)}(\BZ_p)$$
 such that $V_U \subseteq Ker (\varphi_U)$. Note that for a fixed $d$ there
is an upper bound to the order of finite cyclic subgroups  of
$GL_d(\BZ_p)$. Indeed
 by \cite[Thm~5.2]{Marcusbook}  the first  principal congruence subgroup for
$p$ odd and the second principal congruence subgroup for $p=2$ is
uniform and therefore is torsion-free. Thus $[H:Ker (\varphi_U)]$
is bounded by some number depending on $d(F_U)$.

Since $d(F_U)$ is bounded by $d(L) + 1$, the index $[H:Ker
(\varphi_U)]$ is bounded by a number depending on $d(L)$ as well,
so there is some subgroup of finite index $H_1$ in $H$ containing
$L$ and such that $H_1 \subseteq Ker(\varphi_U)$ for every $U$.

 By Lemma \ref{finiteorder}  every automorphism  of finite order of $F_U$
that induces the identity map on $F_U / [F_U,F_U]$ is identity.
Then
 for an open  subgroup $U$ such that $V_U \subseteq H_1$ we have $H_1 /
\widetilde{U} = (V_U / \widetilde{U}) \rtimes
 (H_1 / V_U) = (V_U / \widetilde{U}) \times
 (H_1 / V_U)$, so either $H_1/ \widetilde{U}$ is abelian, in this case $V_U
/ \widetilde{U}$ is procyclic, or $H_1/ \widetilde{U}$ is
non-abelian and the copy of $H_1 / V_U$ is the center of $H_1 /
\widetilde{U}$, in this case $V_U/ \widetilde{U}$ is a
non-procyclic free pro-$p$.  By taking the inverse limit over $U$
we get
 $$
 H_1 / \widetilde{L} = (L / \widetilde{L}) \rtimes (H_1 / L) \simeq (L /
\widetilde{L}) \times (H_1 / L),
 $$
 where $H_1 / L \simeq \BZ_p$. \end{proof}

\begin{lemma}\label{abelianstab} Let $H$ be a
pro-$p$ group acting on a second countable pro-$p$ tree $T$ with
procyclic edge stabilizers and finite rank free abelian pro-$p$
vertex groups. Let $L$ be a finitely generated normal closed
subgroup of $H$ such that $H/L\cong \BZ_p$ and $He=Le$ for every
edge $e\in E(T)$. Assume further that there is a decomposition $H
= L \rtimes \BZ_p$, where $\BZ_p$ has a generator $g$ that fixes
an edge of $T$ and such that  for every vertex $v$ of the tree $T$
we have $H_v = L_v \times \overline{\langle g l_v^{-1} \rangle}$
for some $l_v \in \widetilde{L}= \overline{\langle L_v \rangle}_{v
\in V(T)}$. Then $g$ commutes with $\widetilde{L}$.
\end{lemma}
\begin{proof}
Note that $Le = He$ and $H_e$ procyclic implies that $L_e = 1$ for
every edge $e$ of $T$. Then by Theorem \ref{Kurosh1}
$$
L \simeq  L_0 \amalg (L / \widetilde{L}) \hbox{ and } L_0 =
\coprod_{v \in V_0} {L_v},
$$
where $V_0$ is a subset of the vertex set $V$ and $\widetilde{L}$
is the normal closure of $L_0$ in $L$.

Note that $V_0$ is finite since $L$ is finitely generated. It
follows from the pro-$p$ version of the Kurosh Subgroup Theorem
\cite{melnikov} that
$$
\widetilde{L} = \coprod_{v \in V_0} \coprod_{w} {L_w} = \coprod_{v
\in V_1} L_v,
$$
where $w$ runs over closed set of representatives of the orbit
$Lv$ and $V_1$ is a closed subset of $V$. In particular, this is a
locally constant free pro-$p$ product.

 Observe that $He=Le$ for every edge $e$ implies $Hv=Lv$
for every vertex $v$. Recall that by assumption $H_v$ is abelian
and for every $v \in V$ and for $m \in L_v$ we have $m^{g} =
m^{l_v}$ for some $l_v \in \widetilde{L}$. This implies the
following

\bigskip
{\bf Claim 1.} The action of $g$ via conjugation  on
$\widetilde{L}^{ab} = \prod_{v \in V_1} {L_v}$   is trivial.

\bigskip
Let $M =\langle\widetilde L,g\rangle = \widetilde{L} \rtimes
\overline{\langle g \rangle}$ and $U$ be the closed subgroup of
$H$ generated by $(g^{p^s})^M$.

\bigskip
{\bf Claim 2.} $L_v^h \cap U = 1 \hbox{ for every } v \in V_1, h
\in M.$

\bigskip
Indeed $L_v \cap U \subset U \cap \widetilde{L} =
\overline{[g^{p^s}, \widetilde{L}]}$ and  $\widetilde{L} =
\coprod_{v \in V_1} {L_v}$. Since $L_v$ survives in the
abelianization of $\widetilde{L}$ and  by Claim 1 the group $U$
acts trivially on  the abelianization of $\widetilde{L}$ we get
that $\overline{[g^{p^s}, \widetilde{L}]} \cap L_v = 1$. This
completes the proof of the claim.

\bigskip
In the rest of the proof overlining means image in the quotient
group $M/U$ and should not be confused with the closure. Consider
the group $\overline{M } = M / U$ and the image
$\overline{\widetilde{L}} = \widetilde{L} / (U \cap
\widetilde{L})$ of $\widetilde{L}$ in $\overline{M}$. Since $U$ is
generated by stabilizers of edges (hence of vertices)
$\overline{M}$ acts on the pro-$p$ tree $T / U$ (cf. Theorem
\ref{freeaction}) with vertex stabilizers $L_v U / U \simeq L_v$.
As before $\overline{\widetilde{L}}$ intersects edge stabilizers
trivially and since $\widetilde{L}$ is generated by vertex
stabilizers the same holds for $\overline{\widetilde{L}}$. Hence
$$
\overline{\widetilde{L}} = \coprod_{w \in W} \overline{L}_{w}
$$
where
 $\overline{L}_{w}$ are
some pairwise not conjugated vertex groups  of
$\overline{\widetilde{L}}$. Then by Claim 2 $\overline{L}_w \simeq
L_v$ for some vertex group $L_v$ of $L$.

Write $\bar{g}$ for the image of $g$ in $\overline{M}$. Then by
Claim 1
\begin{equation} \label{abelaction2}
\overline{g} \hbox{ acts trivially (via conjugation) on the
abelianization of }\overline{\widetilde L} .
\end{equation}
 Recall that every $L_v$ is a finite rank free
abelian group. Note that $(U \cap \widetilde{L})$ is contained in
the kernel of the natural epimorphism $L\longrightarrow \prod_{v
\in V_0} {L_v}$ restricted to $\widetilde L$ and so we have a
natural epimorphism $\coprod_{w \in W}
\overline{L}_{w}\longrightarrow \prod_{v \in V_0} {L_v} $, so the
free pro-$p$ product
$$ \overline{\widetilde{L}} = \coprod_{w \in W} \overline{L}_{w}
$$ is also locally constant. Then $\overline{\widetilde{L}}$ has torsion-free
abelianization and by Lemma \ref{centralseries}
$\overline{\widetilde{L}}$ has a torsion-free lower central
series. Then since $\overline{g}$ has finite order we deduce from
Lemma \ref{finiteorder} that $\overline{g}$ acts trivially on
$\overline{\widetilde{L}}$. Taking an inverse limit over $U$ (i.e.
$s$ goes to infinity) we get that $g$ acts trivially on
$\widetilde{L}$ as required.
\end{proof}
\begin{lemma} \label{hyperbolic} Let $H$ be a pro-$p$ group from the class
$\mathcal{L}$ with a finitely generated normal pro-$p$ subgroup
$L$ such that $H / L \simeq \BZ_p$. Then $H$ is abelian.
\end{lemma}
\begin{proof} Let $H = L \rtimes \BZ_p$ be a non-abelian  pro-$p$ group from
the class $\mathcal{L}$ such that $L$ is finitely generated as a
pro-$p$ group. Let  $n$ be the smallest non-negative integer such
that $H \subseteq G_n \in \mathcal{G}_n$ i.e. $n$ is the weight of
$H$. We can assume that our counterexample $H$ is minimal in the
sense that $(n, d(G_n))$ is smallest possible with respect to the
lexicographic order, where $d(G_n)$ is the minimal number of
generators of $G_n$.

The group $G_n$ acts on a second countable pro-$p$ tree $S$  with
procyclic edge stabilizers. By \cite[Lemma~3.11]{horizons} $H$
acts irreducibly on a pro-$p$ subtree $T$ of $S$ with procyclic
edge stabilizers. Let $\widetilde{L} = \overline{ \langle L_v
\rangle}_{v \in V(T)} $.

\medskip
{\bf Claim 1.} We claim that $H$ acts faithfully on $T$.

\medskip
Proof. Indeed if the action of $H$ on $T$ is not faithful since
the edge stabilizers are procyclic we get that the kernel of the
action is $K \simeq \BZ_p$. Then for any $h \in H \setminus K$ the
closed subgroup of $H$ generated by $K$ and $h$ is metabelian. But
any soluble pro-$p$  group from the class $\mathcal{L}$  is
abelian, see Corollary \ref{soluble}. Thus $K \subseteq Z(H)$, a
contradiction to Corollary \ref{centerless}.

\medskip
This way we can assume from now on that $H$ acts faithfully on
$T$. By Theorem \ref{directproduct} there is an open subgroup
$H_1$ of $H$ containing $L$ such that
$$H_1 /
\widetilde{L} \simeq (L / \widetilde{L}) \times M \hbox{ where }M
\simeq \BZ_p.$$

{\bf Claim 2.} The group $H_1$ acts irreducibly on $T$.

\medskip
Proof. By Claim 1 $H$ acts faithfully irreducibly on $T$. Since
$H_1$ is non-trivial, by \cite[Prop.~3.14]{horizons}  $H_1$ acts
irreducibly on $T$. The claim is proved.

\medskip
By Corollary \ref{virtualabelian} every virtually abelian pro-$p$
group from the class $\mathcal{L}$ is abelian, so $H_1$ is not
abelian. Then without loss of generality we can assume  that $H =
H_1$.

The pro-$p$ group $H/ \widetilde{L}$ acts on $T / \widetilde{L}$
and by  \cite[Lemma~3.11]{horizons} contains a minimal $H
/\widetilde{L}$ invariant subtree $T_1$ i.e. $H / \widetilde{L} $
acts irreducibly on the pro-$p$ tree $T_1$.

Let $N$ be the kernel of the action of $H / \widetilde{L}$ on
$T_1$ and so the quotient group $B = (H / \widetilde{L})/ N $ acts
irreducibly and faithfully on $T_1$. By
\cite[Lemma~3.16]{horizons} every non-trivial abelian normal
subgroup $A$ of $B$ is isomorphic to $\BZ_p$ and $C_B(A)$ is a
free pro-$p$ group, hence procyclic. Let $A$ be the image of $M$
in $B$. If $A$ is trivial we get that $M$ acts trivially on $T_1$.

Suppose that $A$ is non-trivial. Then
 $A \simeq \BZ_p \simeq B $.
On the other hand $L / \widetilde{L}$ acts freely on $T_1$, so $(L
/ \widetilde{L}) \cap N = 1$. Then the image of $L /
\widetilde{L}$ in $B  = \BZ_p$ is isomorphic to $L /
\widetilde{L}$, so $L / \widetilde{L}$ is either trivial or
$\BZ_p$. It follows that either $H / \widetilde{L} \simeq \BZ_p^2$
or $H / \widetilde{L} \simeq M \simeq \BZ_p$ and $L =
\widetilde{L}$. In the first case since $B \simeq \BZ_p$  by
changing $M$ with another copy of $\BZ_p$ inside $H
/\widetilde{L}$ we can  reduce to the case when $M \subseteq N$.
 Note that we have proved that
$$
\hbox{either }M \hbox{ acts trivially  on } T_1 \hbox{ or } L =
\widetilde{L}.
$$

\bigskip
{\bf Case 1.} Suppose that $L \not= \widetilde{L}$ and  $M$ acts
trivially on $T_1$. Let $C$ be a connected component of the full
preimage $S_1$ of $T_1$ in $T$.

\bigskip
{\bf Claim 3.} The (set-wise) stabilizers $H_2:=Stab_H(C)$ and
$L_2:=Stab_L(C)$   of $C$ are finitely generated.

\medskip
Proof. Note that $S_1$ is $H$-invariant and that $C$ coincides
with a connected component of the full preimage of
$T_1/(H/\widetilde L)\subseteq T/H$ in $T$. Indeed, if $C$ is
contained properly in a connected component $C_0$ of the full
preimage of $T_1/(H/\widetilde L)$ in $T$ then the image of $C_0$
in $T/\widetilde L$ contains $T_1$ properly contradicting that the
image of $C_0$ is $T_1/(H/\widetilde L)$. Similarly  $C$ coincides
with a connected component of the full preimage of
$T_1/(H/\widetilde L)=T_1/(L/\widetilde L)\subseteq T/L$ in $T$.

\medskip
Let $U$ be an open normal subgroup of $H$ and $V = U \cap L$. Then
$L/\tilde V$ has bounded edge stabilizers. Let $\Sigma(V)$ denote
the set of all finite subgroups $K\neq 1$ of $L/ \tilde V$. Since
$L/\tilde V$ is finitely generated having open free subgroup $V/
\tilde V$, by \cite[Lemma~8]{bulletin}  there is a finite subset
$S$ of $\Sigma(V)$ with $\Sigma(V) =\{K^g\mid K\in S, \ g\in
L/\tilde V\}$. Therefore the subset $T_{\Sigma(V)} :=\{m\in
T/\tilde V\mid \exists L\in\Sigma, m\in (T/\tilde V)^L\}$, which
is the union of all subtrees of fixed points $(T/\tilde V)^K$ for
subgroups $K\in\Sigma(V)$ can be represented in the form
$T_{\Sigma(V)}=\bigcup_{L\in S}(T/\tilde V)^K L/\tilde V$ and is
hence a closed $H$-invariant subgraph of $T/\tilde V$. Therefore
by \cite[Prop., p.486]{open}, the quotient graph $D(V)$ obtained
by collapsing each connected component of $T_{\Sigma(V)}$ to a
vertex is simply connected and hence is a pro-$p$ tree on which
$L/\tilde V$ acts with trivial edge stabilizers.

Since $M$ acts trivially on $T_1$ we have $L_e = 1$ for every $e
\in C$. Indeed
 for every $e \in E(C)$ we have
 $S_0 e \subseteq \widetilde{L} e \subseteq L e $ where $S_0$ is a procyclic
subgroup of $H$ that maps surjectively to $M$
 by
 the canonical map $H \to H / \widetilde{L}$.
 Then
 $H e = (L \rtimes S_0) e = Le$ and so $H_e / L_e \simeq \BZ_p$ for every $e
\in E(C)$.
 Since $H_e$ is procyclic $L_e = 1$ for all $e\in E(C)$.
But by Lemma \ref{component} $H_e=(H_{2})_{e}$ for every $e \in
E(C)$, so $L\cap H_2$ is of infinite index in $H_2$ i.e. $H_2 / (L
\cap H_2) \simeq \BZ_p$.

Put $L_{2V}=L_2/(L_2\cap \tilde V)$. Denote by $C_V$ the image of
$C$ in $T/\tilde V$. Since $L_e = 1$ for all $e \in E(C)$ we
obtain that $(L/\tilde V)_e=1$ for all $e\in E(C_V)$. Then by
Theorem \ref{Kurosh1} applied for $L_{2V}$ acting on the pro-$p$
tree $C_V$
$$
L_{2V}=  (\coprod_{v \in W} (L/\tilde V)_v) \coprod (L /
\widetilde{L})
$$ for some $W \subseteq V(C_V)$, where we used that
$$(L_{2V})_v=(L/\tilde V)_v \hbox{  and }  L /
\widetilde{L}=L_{2V} / \widetilde{L_{2V}}.$$  Indeed for every
$v\in C$ one has $L_v=(L_2)_v$ since $L_2$ is the stabilizer of
the connected component $C$, taking the images of these
stabilizers in $T/\tilde V$ one gets the first equality. For the
second equality note that the definition of $L_2$ implies that $L
= L_2 \tilde{L}$ and by the definition of $C$ the canonical map $T
\to T/ \widetilde{L}$ sends $C$ surjectively to $T_1$. By Lemma
\ref{component} $T_1 = C / (H_2 \cap \widetilde{L})$ and since
$T_1$ is a pro-$p$ tree, by Lemma \ref{converse} $\widetilde L
\cap L_2 =\widetilde L\cap H_2=\widetilde{L\cap H_2} =
\widetilde{L_2}$, where $\widetilde{L_2}$ is the closed subgroup
of $L_2 = L \cap H_2$ generated by stabilizers of vertices in $C$.
Then $$L / \tilde{L} = (L_2 \tilde{L}) / \tilde{L} = L_2 / (L_2
\cap \tilde{L}) = L_2 / \widetilde{L_2}.$$ Let $G = L_{2V} \cap
\tilde{L} / \tilde{V} = (L_2 \tilde{V} \cap \tilde{L}) / \tilde{V}
= (L_2 \cap \tilde{L}) \tilde{V} / \tilde{V} = \widetilde{L_2}
\tilde{V} / \tilde{V}$. Since $T_1 = C_V / (L_{2V} \cap \tilde{L}
/ \tilde{V})$ is a pro-$p$ tree we get that $G = \tilde{G}$, so
$\widetilde{L_{2V}} = \widetilde{L_2} \tilde{V} / \tilde{V}$. Then
$$L_{2V} / \widetilde{L_{2V}} = ({L_2} \tilde{V} / \tilde{V}) /
(\widetilde{L_2} \tilde{V} / \tilde{V}) = L_2 / \widetilde{L_2}.$$

Observe furthermore that $W$ can be chosen any $\delta(C_V /
L_{2V})$ where $\delta : C_V / L_{2V} \to C_V$ is a continious
section of the canonical projection $C_V \to C_V / L_{2V}$.

Note that the collapsing of the connected components of
$T_{\Sigma(V)}$ does not affect $C_V$, i.e. we can denote by the
same letter the isomorphic image of it in $D(V)$. Thus we have the
following commutative diagram

$$\xymatrix{C_V\ar[r]\ar[d]&T/\tilde V\ar[r]\ar[d]&D(V)\ar[d]\\
            T_1=C_V/( \widetilde{L_2} \tilde{V} /\tilde
V)\ar[r]\ar[d]&T/\tilde L\ar[r]\ar[d]&D(V)/(\tilde L/\tilde V)\ar[d]\\
            C_V/L_{2V}=C/L_2=T_1/(L/\tilde L)\ar[r]&T/L=(T/\tilde V)/L/\tilde
            V\ar[r]&D(V)/(L / \tilde{V})\cr}$$

\medskip

\noindent where all composite maps from left to right are
injections (the lower one follows from the middle one). Then by
Theorem \ref{Kurosh1} applied for the action of $ L / \tilde{V}$
on $D(V)$
$$
L/\tilde V \cong (\coprod_{v \in V_0} (L/\tilde V)_v)\coprod_{v\in
W} (L/\tilde V)_v \coprod (L / \widetilde{L})=(\coprod_{v \in V_0}
(L/\tilde V)_v)\amalg L_{2V},
$$
where $W \cup V_0=\mu(D(V)  / (L / \tilde{V}))$ with $\mu : D(V)
/ (L / \tilde{V}) \to D(V)$ to be  an extension of $\delta$ to a
continuous section of the projection $D(V) \to D(V)  / (L /
\tilde{V})$ (see \cite[Exer.~5.6.8]{PavelRibesbook}). It follows
that the number of generators of $L_{2V}$ does not exceed the
number of generators of $L/\tilde V$ for every $V$ and so the
number of generators of $L_{2}$ does not exceed the number of
generators of $L$.

Since $H_2/L_2$ is procyclic, $H_2$ is finitely generated as well.

Finally we observe that by going down to a subgroup of finite
index of $H$ if necessary we can assume  that $H_2$ is not
abelian. Indeed if $H_2$ is abelian since $H_2 \tilde{L} = H$ we
get that $H / \tilde{L} = (L / \tilde{L}) \times M$ is abelian, so
$L / \tilde{L} =\BZ_p$. Since $L$ is not soluble (see Corollary
\ref{soluble}) by Theorem \ref{actions} there is a  free
non-abelian pro-$p$ subgroup $F$ of $L$ acting freely on $T$.
Since $F = F / \tilde{F}$ is the inverse limit of $K / \tilde{K}$
where $K$ runs through the open subgroups of $L$ that contain $F$
we get that for some $K$ the free pro-$p$ group $K / \tilde{K}$ is
not procyclic. Then by substituting $L$ with $K$ if necessary we
get the desired property. This completes the proof of the claim.

\bigskip
Since $M$ fixes all edges of $T_1$
 for every $e \in E(C)$ we have
 $S_0 e \subseteq \widetilde{L} e \subseteq L e $ where $S_0$ is a procyclic
subgroup of $H$ that maps surjectively to $M$
 by
 the canonical map $H \to H / \widetilde{L}$.
 Then
 $H e = (L \rtimes S_0) e = Le$ and so $H_e / L_e \simeq \BZ_p$ for every $e
\in E(C)$.
 Since $H_e$ is procyclic $L_e = 1$ for all $e\in E(C)$.
But by Lemma \ref{component} $H_e=(H_{2})_{e}$ for every $e \in
E(C)$, so $L\cap H_2$ is of infinite index in $H_2$ i.e. $H_2 / (L
\cap H_2) \simeq \BZ_p$.

 By the definition of $C$ the canonical map $T \to T/ \widetilde{L}$ sends
$C$ surjectively to $T_1$. By Lemma \ref{component} $T_1 = C /
(H_2 \cap \widetilde{L})$ and since $T_1$ is a pro-$p$ tree, by
Theorem \ref{converse} $\widetilde L\cap H_2=\widetilde{L\cap
H_2}$, where $\widetilde{L\cap H_2}$ is the closed subgroup of $L
\cap H_2$ generated by stabilizers of vertices in $C$. Hence
replacing $H$ by $H_2$, $L$ by $L_2$ and $T$ by $C$ we may assume
that
\begin{equation} \label{important} T /
\widetilde L=T_1 \end{equation} except that  $H / \widetilde{L} =
(L / \widetilde{L}) \times \BZ_p$ might not hold. But by Theorem
\ref{directproduct} it is sufficient to change $H$ to a subgroup
of finite index that contains $L$ to repair this
 property and by Claim 2 this does not affect the fact that $H /
\widetilde{L}$ acts irreducibly on $T_1$.
 So from now on we can assume that \begin{equation} \label{novo99}  H / \widetilde{L} = (L / \widetilde{L})
\times \BZ_p. \end{equation}

Thus we have $L_e = 1$ for every $e \in E(T)$. Then by Theorem
\ref{Kurosh1}
$$
L \cong (\coprod_{v \in V_0} L_v) \coprod (L / \widetilde{L}).
$$
Since $L$ is finitely generated  all $L_v$ are finitely generated.

Fix $v \in V(T)$. Since $M$ fixes all vertices of $T_1$ we have
$Lv = Hv$. Since $H_v / L_v \simeq \BZ_p$, we have $H_v = L_v
\rtimes \BZ_p$. By the minimality of $n$ we deduce that $$H_v
\hbox{ is abelian for all } v \in V_0.$$ This together with the
fact that $M$ acts trivially on $T_1$ and (\ref{novo99}) implies
that $H$ and $L$ satisfy the assumptions of Lemma
\ref{abelianstab}  and by Lemma \ref{abelianstab} we deduce that
there is some $g \in H \setminus L$ that fixes an edge of $T$ and
$$g \hbox{ commutes with }\widetilde{L}.$$
Let $l$ be an element of $L$ such that  the image in $H /
\widetilde{L}$ of the closed subgroup of $H$ generated by $lg$  is
$M$.

Since $M$ fixes  $T_1$ we have for a fixed vertex $v$ of $T$ that
$$(lg) v = l_v v$$ for some $l_v \in \widetilde{L}$. Thus $H_v =
L_v \rtimes \overline{\langle l_v^{-1} l g \rangle}$ . Note that
by Lemma \ref{abelianstab} $g$ acts trivially (via conjugation) on
$\widetilde{L}$ , hence $g$ acts trivially on $L_v$. On the other
hand the fact that $H_v$ is abelian implies that $l_v^{-1} lg$
commutes with $L_v$. Then $l_v^{-1} l$ commutes with $L_v$. Since
$\widetilde{L}$ is a free pro-$p$ product of pro-$p$ groups (one
of which is $L_v$) we deduce that $l_v^{-1} l  \in L_v$ because
every free factor is self-centralized (see Corollary 4.4 (a) in
\cite{horizons}), so $l_v v = l v$. Then
$$g v = l^{-1} l_v v = v,$$ so $g$ fixes every vertex of $T$. But the action
of $H$ on $T$ is faithful, a contradiction.

\bigskip
{\bf Case 2.} Suppose that $L = \widetilde{L}$. Let $U$ be an open
normal subgroup of $H$ and $V = U \cap L$.
 Suppose that $V = \widetilde{V}$ for all $U$ (otherwise we can continue as
in Case 1 substituting $L$ with $V$ and $H$ with $U$). Then $L/ V
= L / \widetilde{V}$ is finite and so stabilizes some vertex of
the pro-$p$ tree $T / \widetilde{V}$.

 Since $\cap \widetilde{V} = \cap V = 1$ the inverse limit of $L /
\widetilde{V}$ over $\widetilde{V}$ is $L$. On the other hand this
inverse limit stabilizes a vertex of $T$ as every $L /
\widetilde{V}$ stabilizes some vertex of $T / \widetilde{V}$. Thus
$L$ is in a conjugate of $G_{n-1}$ or $A$. In the second case $L $
is abelian, hence $H$ is soluble and by Corollary \ref{soluble}
$H$ is  abelian, a contradiction. In the first case conjugating if
necessary we can suppose that $L \subseteq G_{n-1}$. Consider an
epimorphism $$\varphi : G_n = G_{n-1} \amalg_C A \to G_{n-1}
\amalg_C \overline{A}$$ with free kernel, $\varphi$ induces the
identity map on $G_{n-1}$ and an epimorphism $A \to \overline{A}$
where $\overline{A}$ is a free abelian pro-$p$ group of rank $d(A)
- 1$. Thus $ker (\varphi) \cap L = 1$. Then $\varphi(H) =
\varphi(L) \rtimes \BZ_p$ or $[\varphi(H) : \varphi(L)] < \infty$.
In the first case $H \simeq \varphi(H)$ is a pro-$p$ subgroup of
$G_{n-1} \amalg_C \overline{A}$. Note that the weight  of $G_{n-1}
\amalg_C \overline{A}$ is at most $n$ and that $d(G_{n-1} \amalg_C
\overline{A}) < d(G_{n-1} \amalg_C {A})$, in contradiction to the
minimality of $(n, d(G_n))$.
 In the second case $ker (\varphi) \cap H = H_0$ is infinite procyclic, and
$H \subseteq N_{G_n}(H_0)$.
 By Theorem \ref{centralizers} $N_{G_n}(H_0) = C_{G_n}(H_0)$ is abelian, so
$H$ is abelian, a contradiction.
\end{proof}
\begin{theorem} \label{normalabelian} Let $H$ be a pro-$p$ group from the
class $\mathcal{L}$ with a non-trivial  finitely generated normal
pro-$p$ subgroup $N$ of infinite index. Then $H$ is abelian.
\end{theorem}
\begin{proof}
Since $N$ is a pro-$p$ group from the class $\mathcal{L}$ by
Corollary \ref{infinity} $N$ is of type $FP_{\infty}$. Then by the
main result of \cite{ThomasPavel} there is a finite index subgroup
$G_0$ of $H$ such that $G_0$ contains $N$ and $cd(G_0/ N) <
\infty$. In particular $G_0 / N$ is non-trivial and torsion-free.

Let $g \in G_0 \setminus N$. The group $T = N . \overline{\langle
g \rangle}$ is a pro-$p$ group from the class $\mathcal{L}$,  so
by Corollary \ref{infinity} is  of type $FP_{\infty}$ and $T/ N
\simeq \BZ_p$. By Lemma \ref{hyperbolic} $T$ is abelian and since
there was no restriction on the choice of $g$ in $G_0 \setminus N$
we get that $N \subseteq Z(G_0)$. By Corollary \ref{centerless}
$G_0$ is abelian and by Corollary \ref{virtualabelian}  $H$ is
abelian.
\end{proof}
\begin{cor} Let $H$ be a pro-$p$ group from the class $\mathcal{L}$ with a
free pro-$p$ subgroup $F$ of rank $2$. Then $F = N_H(F)$.
\end{cor}
\begin{proof}
Suppose that $F \not= N_H(F)$. Let $g \in N_H(F) \setminus F$ and
define $G$ as the pro-$p$ subgroup of $H$ generated by $F$ and
$g$. If $G / F$ is infinite then $G$ is a non-abelian  pro-$p$
group from the class $\mathcal{L}$ with a finitely generated
normal pro-$p$ subgroup of infinite index, in contradiction to
Theorem \ref{normalabelian}. Thus $F$ is open in $G$ and by
Serre's result \cite{S-65} $G$ is a free pro-$p$ group. Since $F
\not= G$ the Schreier formula yields  $2= d(F) > d(G)$, so $G$ is
procyclic and $F$ is procyclic, a contradiction.
\end{proof}
\begin{theorem} Let $G$ be a pro-$p$ group from the class $\mathcal{L}$ and
$H$ a non-trivial finitely generated subgroup of $G$. Then
$[N_G(H):H]$ is finite unless $N_G(H)$ is abelian.\end{theorem}
\begin{proof} Suppose on the contrary that $N_G(H)$ is non-abelian and
$[N_G(H):H]$ is infinite. Put $N=N_G(H)$ and suppose that $N/ H$
is not torsion. Let $t \in N \setminus H$. Suppose first that $t$
has infinite order modulo $H$ (note that we have supposed that
such $t$ exists). Then the closed subgroup of $N$ generated by $H$
and $t$ has a normal closed finitely generated subgroup $H$ and so
by Lemma \ref{hyperbolic} is abelian. In particular $H$ is
abelian.  If $t$ has finite order modulo $H$ then by Corollary
\ref{virtualabelian} the closed subgroup of $N$ generated by $H$
and $t$ is abelian, so $[t, H]= 1$. Since $H$ is non-trivial, by
transitive commutativity any two elements of $N$ commute, so $N$
is abelian, a contradiction.

Thus we can suppose from now on  that $N/H$ is an infinite torsion
pro-$p$ group. Since $H$ is a pro-$p$ group from the class
$\mathcal{L}$ by Corollary \ref{infinity} $H$ is of type
$FP_{\infty}$. By Lemma \ref{torsion} every group from the class
$\mathcal{L}$ has finite cohomological dimension and hence every
subgroup of a group from the class $\mathcal{L}$ has finite
cohomological dimension. In particular $cd(N) < \infty$.
 Then by the main result of \cite{ThomasPavel} there is a finite index
subgroup $G_0$ of $N$ such that $G_0$ contains $H$ and $cd(G_0/ H)
< \infty$. In particular $G_0 / H$ is torsion-free and hence $G_0
= H$, a contradiction to the fact that $N/H$ is infinite. This
completes the proof.\end{proof}
\section{$2$-generated pro-$p$ groups in the class $\mathcal{L}$}\label{two
generated}
\begin{lemma} \label{intersection} Let $Q$ be a non-trivial free abelian
pro-$p$ group of finite rank and $Y$ be the set of all
epimorphisms of pro-$p$ groups $\varphi : Q \to Q_1 = \BZ_p$. Let
$\hat{\varphi}$ be the continuous ring homomorphism $\BF_p[[Q]]
\to \BF_p[[Q_1]]$ induced by $\varphi$. Then $\cap_{\varphi \in Y}
Ker(\hat{\varphi}) = 0$.
\end{lemma}
\begin{proof} We prove by induction on $n$ the following more general
statement. Let $R$ be  a pro-$p$ ring of characteristic $p$,
$R[[t_1, \ldots, t_n]]$ be the commutative ring of formal power
series and $Q_{R,n}$ be the closed group generated by $\{ 1 + t_i
\}_{1 \leq i\leq n}$. Consider the set $W$ of all homomorphism of
pro-$p$ rings $\widehat{\varphi} : R[[t_1, \ldots, t_n]] \to
R[[t_1]]$ that induce an epimorphsim of pro-$p$ groups $\varphi :
Q_{R, n} \to Q_{R,1}$. We claim that $$\cap_ {\widehat{\varphi}
\in W } ker \widehat{\varphi}  = 0.$$ By writing $R[[t_1, \ldots,
t_n]] $ as $R[[t_n]][[t_1, \ldots, t_{n-1}]]$ we see that to prove
the inductive step is sufficient to consider the case $n = 2$.

Assume from now that $n = 2$ and define $\theta_i  : R[[t_1, t_2]]
\to R[[t_1]]$ by $\theta_i(t_1) = t_1$ and $\theta_i(t_2) =
t_1^{p^i}$. Thus $\theta_i(1 + t_2) = (1 + t_1)^{p^i}$ and
$\theta_i$ sends $Q_{R, 2}$ surjectively to $Q_{R, 1}$. We prove
by induction on $j$ the following claim : for $ker( \theta_i) =
(t_2 - t_1^{p^i}) R[[t_1, t_2]]$
\begin{equation} \label{999}
\cap_{1 \leq i \leq j} ker (\theta_i) = \prod_{1 \leq i \leq j}
ker (\theta_i)
\end{equation}
Indeed if $f$ is in the left hand side of (\ref{999}), then by
induction $f = (\prod_{1 \leq i \leq j-1} (t_2 - t_1^{p^i})) f_0$
for some $f_0 \in  R[[t_1, t_2]]$. Since $\theta_j(f) = 0$,
$\theta_j (t_2 -  t_1^{p^i}) \not= 0$ for $ 1 \leq i \leq j-1$ and
$R[[t_1]]$ is a domain we deduce that $f_0 \in Ker(\theta_j) =
(t_2 - t_1^{p^j}) R[[t_1, t_2]]$.  This completes the proof of the
claim.

Finally let $I$ be the augmentation ideal of $R[[t_1, t_2]]$. By
the claim $\cap_{1 \leq i \leq j} ker (\theta_i) \subseteq I^j$
and $\cap_{j \geq 1} I^j = 0$. Thus $\cap_{i \geq 1} ker
(\theta_i) = 0$.
 \end{proof}
\begin{lemma} \label{nakayama} Let $1 \to F \to G \to Q \to 1$ be a short
exact sequence of pro-$p$ groups with $Q$ non-trivial pro-$p$
abelian torsion-free, $G$ finitely generated and $F$ free pro-$p$,
1-generated as a closed normal subgroup of $G$. Then either $G$
has a pro-$p$ subgroup $N$ such that $G / N \simeq \BZ_p$ and $N$
is finitely generated or $F / F' F^p \simeq \BF_p[[Q]]$ and  so
there is a free pro-$p$ subgroup $F_0$ of $G$ such that $F \subset
F_0$, $F_0 / F \simeq \BZ_p$ and $G / F_0$ is torsion-free.
\end{lemma}
\begin{proof} Since $F$ is generated by one element as a closed normal
subgroup of $G$ we have that $V = F / F' F^p$ is a cyclic pro-$p$
$\BF_p[[Q]]$-module, where $Q$ acts by conjugation.

Suppose first that $V$ is not free as a pro-$p$
$\BF_p[[Q]]$-module. Then $V \simeq \BF_p[[Q]] / I$ where $I$ is
some closed ideal in $\BF_p[[Q]]$. Let $\varphi : Q \to Q_1 =
\BZ_p$ be a surjective homomorphism of pro-$p$ groups and
$\hat{\varphi}$ be the continuous ring homomorphism $\BF_p[[Q]]
\to \BF_p[[Q_1]]$ induced by $\varphi$. By Lemma
\ref{intersection}  $\cap Ker(\hat{\varphi}) = 0$ when $\varphi$
runs through all such epimorphisms.  In particular there is
$\varphi$ such that $\hat{\varphi}(I) \not= 0$. Then $V
\widehat{\otimes}_{\BF_p[[Ker(\varphi)]]} \BF_p  \simeq
(\BF_p[[Q]]/ I) \widehat{\otimes}_{\BF_p[[Ker(\varphi)]]} \BF_p
\simeq \BF_p[[Q_1]] / \hat{\varphi}(I)$ is finite and by
Nakayama's lemma $V$ is finitely generated as a pro-$p$
$\BF_p[[Ker(\varphi)]]$-module. Then the preimage $N$ of
$ker(\varphi)$ in $G$ is finitely generated as a pro-$p$ group and
$G / N \simeq \BZ_p$.

Suppose now that $V \simeq \BF_p[[Q]]$. Take any element $q \in Q$
such that the pro-$p$ subgroup of $Q$ generated by $q$ is a direct
factor in $Q$. Let $g$ be an element of the preimage of $q$ in
$G$. Then the pro-$p$ subgroup $F_0$ of $G$ generated by $F$ and
$g$ is free pro-$p$. Indeed $V \simeq \BF_p[[Q]]$ is also a free
$\BF_p[[\overline{\langle q\rangle}]]$-module with the basis $Z$,
where $Z \times \overline{\langle g \rangle} = Q$. Lifting $Z$ to
a closed subset $\widetilde Z$ of $F$ and translating (i.e.
conjugating) it by $\overline{\langle g\rangle}$ we obtain a
closed $\overline{\langle g\rangle}$-invariant basis $X$ of $F$ as
a free pro-$p$ group where the image of $X$ in $V$ is $\{ Q \}$.
Then $F_0$ is a free pro-$p$ group with a basis $\widetilde Z \cup
\{ g \}$.
\end{proof}
\begin{theorem} Every $2$-generated pro-$p$ group from the class
$\mathcal{L}$ is either free pro-$p$ or abelian.
\end{theorem}
\begin{proof} Let $H$ be a $2$-generated pro-$p$ subgroup of $G_n \in
\mathcal{G}_n$ such that $H$ is not procyclic. We prove by
induction on the minimal number of generators of $G_n$ that $H$ is
either free pro-$p$ or abelian. The induction starts since the
minimal number of generators for $G_n$, when $H$ exists, is 2 and
in this case $G_n$ is either free pro-$p$ or abelian.

The case $n = 0$ is obvious since free pro-$p$ groups  have only
free pro-$p$ subgroups. Thus we can assume that $n \geq 1$. Let
$G_n = G_{n-1} \amalg_C A$, where $A = C \times B \simeq \BZ_p^m$
and $C \simeq \BZ_p$. Let $D \simeq \BZ_p$ be a direct summand of
$B$ and $M = \overline{\langle \cup_{g \in G_n} D^g \rangle}$.
Thus $M$ is the kernel of the epimorphism of pro-$p$ groups
$\varphi : G_n = G_{n-1} \amalg_C A \to G_{n-1} \amalg_C (A / D)
=L$ that is the identity map on $G_{n-1}$ and the canonical
projection $A \to A / D$ on $A$. By Theorem \ref{strongKurosh} $M$
is free pro-$p$. Note that $L$ is a pro-$p$ group from the class
$\mathcal{L}$ with less generators than $G_n$. By induction
$\varphi(H)$ is either free pro-$p$ or abelian. Note that by Lemma
\ref{torsion} $\varphi(H)$ is torsion-free.

1. If $\varphi(H)$ is free pro-$p$ of rank 2, since $H$ is
2-generated $H \simeq \varphi(H)$ and we are done.

2. If $\varphi(H) \simeq \BZ_p$ we get that $H$ is an extension of
the free pro-$p$ group $M_0 = M \cap H$ by $\BZ_p$. We claim that
every 2-generated pro-$p$ group from the class $\mathcal{L}$ that
is free-by-$\BZ_p$ is either free pro-$p$ or abelian. Since $H$ is
2-generated $M_0$ is generated as a normal closed subgroup of $H$
by just one element. Then for $V = M_0 / \overline{M_0^p [M_0,
M_0]}$ we have that $V$ is a cyclic $\BF_p[[Q]]$-module via
conjugation, where $Q = H / M_0 \simeq \BZ_p $. Suppose that $H$
is not abelian. By Theorem \ref{hyperbolic} $M_0$ is not finitely
generated as a pro-$p$ group, so $V$ is infinite and hence $V
\simeq \BF_p[[Q]]$. This implies together with Lemma
\ref{nakayama} that $H$ is a free pro-$p$ group.

3. If $\varphi(H) \simeq \BZ_p \times \BZ_p$ we have that $H$ is
an extension of the free pro-$p$ group $M_0 = M \cap H$ by $\BZ_p
\times \BZ_p$. Since $H$ is 2-generated, $M_0$ is generated as a
normal closed subgroup of $H$ by just one element, namely the
commutator of two generators of $H$. Then by Lemma \ref{nakayama}
either there is a normal finitely generated pro-$p$ subgroup of
infinite index in $H$, a contradiction with Theorem
\ref{normalabelian}, or  there is a normal free pro-$p$ subgroup
$F_0$ of $H$ such that $F_0 / M_0 \simeq \BZ_p$. In the last case
$H / F_0 \simeq \BZ_p$ and we can argue exactly as in the case 2.

4. If $\varphi(H)$ is the trivial group then $H$ is a pro-$p$
subgroup of the free pro-$p$ group $M$, so is free pro-$p$.
\end{proof}

\section{Euler characteristic}

By definition a pro-$p$ group $G$ from the class $\mathcal{L}$  is
a closed subgroup of some $G_n \in \mathcal{G}_n$, hence $cd(G)
\leq cd(G_n) < \infty$. Since $G$ is of type $FP_{\infty}$, see
Corollary \ref{infinity},  it has a well-defined Euler
characteristic $\chi(G) = \sum_{0 \leq i \leq cd(G)} (-1)^i
dim_{\BF_p} H_i(G, \BF_p)$.

\begin{theorem} \label{Eulerchar} Every pro-$p$ group $G$ from the class
$\mathcal{L}$ has Euler characteristic $\chi(G) \leq 0$.
\end{theorem}

\begin{proof}
Let $N$ be a closed normal subgroup of $G$ such that $Q = G / N$
is torsion-free nilpotent and $N$ is free pro-$p$ (see Proposition
\ref{freenilpotent}). Let
$$
{\mathcal P} : 0 \to P_m \to P_{m-1} \to \ldots \to P_1 \to P_0
\to \BZ_p \to 0
$$
be a free resolution of the trivial right $\BZ_p[[G]]$-module
$\BZ_p$ with $P_i$ finitely generated, say of rank $\alpha_i$, for
all $i$.  Note that
$$
H_i({\mathcal P} \widehat{\otimes}_{\BZ_p[[N]]} \BF_p) = H_i(N,
\BF_p) \hbox{ for } i \geq 1.
$$
Since $N$ is a free pro-$p$ group
\begin{equation} \label{homology10}
H_i({\mathcal P} \widehat{\otimes}_{\BZ_p[[N]]} \BF_p) = 0 \hbox{
for } i \not= 1
\end{equation}
and
$$
H_1({\mathcal P} \widehat{\otimes}_{\BZ_p[[N]]} \BF_p) = N /
\overline{[N,N] N^p} = V.
$$

Since $Q$ is a torsion-free nilpotent pro-$p$ group and hence of
finite rank, $\BF_p[[Q]]$ is a right and left Noetherian ring
without zero divisors \cite[Cor. 7.25]{Marcusbook}.
 Then $\BF_p[[Q]]$ is an Ore ring and has a classical ring of quotients,
denoted by $K$. Note that  the abstract tensor
 product $\otimes_{\BF_p[[Q]]} K$ is an exact functor and $$V
\otimes_{\BF_p[[Q]]} K
 \simeq H_1({\mathcal P} \widehat{\otimes}_{\BZ_p[[N]]} \BF_p)
\otimes_{\BF_p[[Q]]} K
 \simeq H_1(({\mathcal P} \widehat{\otimes}_{\BZ_p[[N]]} \BF_p)
\otimes_{\BF_p[[Q]]} K) \simeq K^{z} $$ for some non-negative
integer $z$. Then using that $(P_i \widehat{\otimes}_{\BZ_p[[N]]}
\BF_p) \otimes_{\BF_p[[Q]]} K \simeq K^{\alpha_i}$ and
(\ref{homology10}) we get
 $$
 \chi(G) = \sum_i (-1)^i \alpha_i =
 \sum_i (-1)^i dim_{K} H_i (({\mathcal P} \widehat{\otimes}_{\BZ_p[[N]]}
\BF_p) \otimes_{\BF_p[[Q]]} K).$$
 Note that since $- \otimes_{\BF_p[[Q]]} K$ is an exact functor it commutes
with homology, hence
 $$H_i (({\mathcal P} \widehat{\otimes}_{\BZ_p[[N]]} \BF_p)
\otimes_{\BF_p[[Q]]} K) \simeq
 (H_i ({\mathcal P} \widehat{\otimes}_{\BZ_p[[N]]} \BF_p))
\otimes_{\BF_p[[Q]]} K.
 $$ Then
 $$ \chi(G) = \sum_i (-1)^i dim_{K} ((H_i ({\mathcal P}
\widehat{\otimes}_{\BZ_p[[N]]} \BF_p)) \otimes_{\BF_p[[Q]]} K ) =
$$ $$
 -  dim_{K}( H_1 ({\mathcal P} \widehat{\otimes}_{\BZ_p[[N]]} \BF_p)
\otimes_{\BF_p[[Q]]} K) = -z \leq 0.
 $$
\end{proof}
\begin{theorem} \label{free-abelian} Let $G$ be a  pro-$p$ group from the
class $\mathcal{L}$ that is free-by-abelian and such that its
Euler characteristic $\chi(G) = 0$. Then $G$ is abelian.
\end{theorem}
\begin{proof} Let $F$ be a closed normal subgroup of $G$ such that $F$ is
free pro-$p$ and $Q = G / F$ is abelian. By going down to a
subgroup of finite index if necessary we can assume that $G / F$
is torsion-free. By the proof of Corollary \ref{Eulerchar}
$\chi(G) = -z$ where $dim_K(V \otimes_{\BF_p[[Q]]} K) = z$, $V  =
H_1(F, \BF_p)$ and $K$ is the abstract field of fractions of
$\BF_p[[Q]]$. Since $\chi(G) = 0$ we have that $z = 0$ and $V
\otimes_{\BF_p[[Q]]} K = 0$. So the annihilator
$ann_{\BF_p[[Q]]}(V) = I$ is non-zero. By the proof of Lemma
\ref{nakayama} there is an epimorphism $\varphi : Q \to Q_1 \simeq
\BZ_p$ such that $\BF_p[[Q]] / I$ is finitely generated as a
pro-$p$ $\BF_p[[Ker(\varphi)]]$-module. Then $V$ is finitely
generated as a pro-$p$ $\BF_p[[Ker(\varphi)]]$-module.

Let $M$ be the preimage of $Ker(\varphi)$ in $G$. Then $M$ is a
normal closed subgroup of $G$, $M$ is finitely generated and $G /
M \simeq \BZ_p$. By Lemma \ref{hyperbolic} $G$ is abelian.
\end{proof}

\section{Open questions}

Here we list  possible properties of the pro-$p$ groups $G$ from
the class $\mathcal{L}$.

1. $def(G) \geq 2$ if every abelian pro-$p$ subgroups of $G$ is
procyclic and $G$ itself is not procyclic, where $def(G) =
dim_{F_p} H^1(G, \BF_p) - dim_{F_p} H^2(G, \BF_p)$ denotes the
deficiency of $G$.

2. Every finitely generated subgroup of a pro-$p$ group from the
class $\mathcal{L}$ is a virtual retract.

3. The Euler characteristic $\chi(G) = 0$ if and only if  $G$ is
abelian. In the particular case when $G$ is free-by-abelian this
holds by Theorem \ref{free-abelian}.

4. The Howson property:  the intersection of any two finitely
generated pro-$p$ subgroups of $G$ is finitely generated.

5. $G$ is residually free pro-$p$.

6. $G$ is residually torsion-free pro-$p$ nilpotent.

\medskip
Note that 5 implies 6. For abstract limit groups a stronger
version of 5 holds as abstract limit groups are fully residually
free. In the abstract case property 4 was proved in \cite{KMRS}, 3
in \cite{desi},  2 in  \cite{Wilton}. Property 1 follows by
induction on the height $h$ of an abstract limit group: one uses
the fact that a freely indecomposable abstract limit group of
height $h \geq 1$ is the fundamental group of a finite graph of
groups with infinite cyclic edge groups such that at least one
vertex group is a non-abelian limit group of height $\leq h - 1$
\cite[Lemma~1.4]{BH}. The definition of height of an abstract
limit group can be found in \cite[Section~1]{BH} and differs from
our notion of weight by allowing almost all surface groups as
height zero limit groups.

\bigskip
{\bf Acknowledgements}

\medskip
The authors thank the referee for the many helpful remarks that
improved the paper.

\end{document}